\documentclass[3p,,times]{elsarticle}

\usepackage{hyperref}
\usepackage{amsfonts}

\usepackage{array,booktabs}

\usepackage{amsmath,amsthm,amsfonts,amssymb,latexsym,epsfig}
\usepackage{hyperref}
\usepackage{color}
\usepackage{graphicx}
\usepackage{multirow}
\usepackage{float}
\usepackage{subfigure}
\usepackage[noend]{algpseudocode}

\usepackage{algorithmicx,algorithm}

\newtheorem{thm}{Theorem}[section]

\newtheorem{lem}[thm]{Lemma}

\newtheorem{exa}[thm]{Example}

\numberwithin{equation}{section}

\begin{document}

\begin{frontmatter}

\title{Eigenvalue embedding of undamped piezoelectric structure system with no-spillover \tnoteref{mytitlenote}}
\tnotetext[mytitlenote]{Supported by Research Foundation of Changsha University of Science and Technology (Grant No. 2019QJCZ051) and Department of Education of Hunan Province (Grant No. 19B028).}
\author[mymainaddress]{Kang Zhao\corref{mycorrespondingauthor}}
\cortext[mycorrespondingauthor]{Corresponding author}
\ead{zhaokangmath@126.com}



\address[mymainaddress]{School of Mathematics and Statistics, Changsha University of Science and Technology, Changsha, 410114, P. R. China}

\begin{abstract}
In this paper, the eigenvalue embedding problem of the undamped piezoelectric structure system with no-spillover (EEP-PS) is considered. It aims to update the original system to a new undamped piezoelectric structure system, such that some eigenpairs are replaced by newly given or measured ones, while the remaining eigenpairs are kept unchanged. A set of parametric solutions to the EEP-PS are provided and the performance of the proposed algorithms are illustrated by numerical examples.
\end{abstract}

\end{frontmatter}

\section{Introduction}
The technology of smart materials and structures especially piezoelectric smart structures has become mature over the last decade. One promising application of piezoelectric structures is the control and suppression of unwanted structural vibrations \cite{Sunar-Ps-1999}. Piezoelectric materials such as lead zirconate titanate are extensively used in vibration damping applications \cite{AAD-2001}. The three-dimensional piezoelectric constitutive law can be written as \cite{MV-book-smart,IT-book-1990}:

\begin{equation}\label{gg-m-1}
\left\{\begin{array}{l}
\{\sigma\}=[c]\{\epsilon\}-[e]^T\{E\},\\
\{D\}=[e]\{\epsilon\}+[\varepsilon]\{E\},\\
\end{array}\right.
\end{equation}
where mechanical variable $\sigma$ and $\epsilon$ denote the stress and the strain, respectively; electrical variable $D$ and $E$ denote the electric displacement and the electric field, respectively. Matrices $[c], [e]$ and $[\varepsilon]$ represent material properties: $[c]$ is the elasticity matrix, $[e]$ is the piezoelectric matrix, $[\varepsilon]$ is the dielectric matrix. The indirect piezoelectric effect is given by the first equation of (\ref{gg-m-1}), while the second equation characterizes the direct piezoelectric effect. Variational principles can be used to establish the finite element equations for piezoelectric structures. Modeling of piezoelectric smart structures by the finite element method and its implementation for active vibration control problems has been presented \cite{Meng-Ps-2006,Naraya-2003}. The global equation of motion governing a undamped structure system with $n$ degrees of freedom can be written as \cite{Tzou-Ps-1990}
\begin{equation}\label{gs-1}
M\begin{bmatrix}
\ddot{u}\\
\ddot{\phi}\\
\end{bmatrix}+K\begin{bmatrix}
u\\
\phi\\
\end{bmatrix}=\begin{bmatrix}
F_u\\
F_{\phi}\\
\end{bmatrix},
\end{equation}
where
\begin{equation}\label{gs-2}
M=\begin{bmatrix}
M_u & 0\\
0 & 0\\
\end{bmatrix},\ \ K=\begin{bmatrix}
K_u & K_{u\phi}\\
K_{u\phi}^T & K_{\phi}\\
\end{bmatrix}\in\mathbb{R}^{n\times n},
\end{equation}
with $M_u, K_u\in\mathbb{R}^{n_u\times n_u}$ and $K_{\phi}\in\mathbb{R}^{n_{\phi}\times n_{\phi}}$ being symmetric, $n=n_u+n_{\phi}$; $u$ denotes structural displacement, $\phi$ denotes electric potential; $M_u$ and $K_u$ are the structural mass and stiffness matrices, respectively, $K_{u\phi}$ is the piezoelectric coupling matrix and $K_{\phi}$ is the dielectric stiffness matrix; $F_u$ denotes the structural load and $F_{\phi}$ denotes the electric load. The  modality of undamped piezoelectric structure (\ref{gs-1}) is characterized by eigenvalues and eigenvectors of its associated generated eigenvalue problem
\begin{equation}\label{gs-3}
P(\lambda)x:=(\lambda M+K)x=0.
\end{equation}
We will refer to (\ref{gs-3}) with $M, K$ of the form (\ref{gs-2}) as an undamped piezoelectric structure system.

Eigenvalue embedding, also known as model updating in some literature, has recently been an active topic, for example, \cite{Qiang-2017,Chu-2007,Chu-2008,Chu-2009,Kang-2018-apm,Lancaster-2008,Friswell-1998}\cite{Carvalho-2006} and the references therein. The main purpose of eigenvalue embedding is to update the original system to a new ne, such that some "troublesome" or "unwanted" eigenvalues and the corresponding eigenvectors are replaces or some desired eigenvalues and eigenvectors are achieved. Among current developments for model updating, the problem of maintaining the remaining or unknown eigenvalues and eigenvectors unchanged is of practical importance, which is known as no-spillover phenomena in literature, see \cite{Carvalho-2006,Chu-2007}.

In this paper, we consider the eigenvalue embedding problem of the undamped piezoelectric smart system with no-spillover, which is to update the original system to a new undamped piezoelectric smart system, such that some nonzero eigenvalues are replaced with some desired ones, while the remaining eigenvalues and eigenvectors are kept unchanged. The eigenvectors corresponding to the updated eigenvalues of the new system can be different from those corresponding to the eigenvalues to be updated of the original system. But there may be some restrictions on them. For example, in the model updating problems for the second order system, it has been shown in \cite{Chu-2009,Chu-2007} that, under some mild conditions, the newly measured eigenvectors must lie in the same subspace as those spanned by eigenvectors corresponding to the eigenvalues to be updated of the original system. Hence, in this paper, we restrict the eigenvalue embedding problem in that the newly given or measured eigenvectors span the same subspaces as the original eigenvectors. The eigenvalue embedding problem for the undamped piezoelectric smart system with no-spillover (EEP-PS) can be stated as follows.

{\bf EEP-PS}: Given an undamped piezoelectric system $(\lambda M+K)x=0$ with $M, K$ of the form (\ref{gs-2}), some of its eigenpairs $\{(\lambda_i, x_i)\}_{i=1}^p$ $(p\ll n_u)$ with $\{\lambda_i\}_{i=1}^p$ being closed under complex conjugate, and a set of $p$ nonzero numbers $\{\tilde{\lambda}_i\}_{i=1}^p$, closed under complex conjugate, update the original system to a new undamped piezoelectric smart system $$(\lambda \widetilde{M}+\widetilde{K})x=0$$ with $\widetilde{M}=\begin{bmatrix}
\widetilde{M}_u & 0\\
0 & 0\\
\end{bmatrix}$ and $\widetilde{K}=\begin{bmatrix}
\widetilde{K}_u & \widetilde{K}_{u\phi}\\
\widetilde{K}_{u\phi}^T & \widetilde{K}_{\phi}\\
\end{bmatrix}$ of the form (\ref{gs-2}), such that the part of eigenvalues $\{\lambda_i\}_{i=1}^p$ of the original system are replaced by $\{\tilde{\lambda}_i\}_{i=1}^p$, the eigenvectors $\{\tilde{x}_i\}_{i=1}^p$ corresponding to $\{\lambda_i\}_{i=1}^p$ of the updated system span the same subspaces as $\{x_i\}_{i=1}^p$, and the remaining $n-p$ eigenvalues and eigenvectors are kept unchanged.

We will provide a set of solutions depending on some parameter matrices to the EEP-PS. Obviously, if $n_{\phi}\geq 1$, the coefficient matrix $M$ in (\ref{gs-2}) is singular, which implies that the matrix polynomial $P(\lambda)$ contain both finite and infinity eigenvalues. Thus, the existing results on the Eigenvalue embedding/model updating of quadratic systems cannot be directly applied to the undamped piezoelectric system. In this paper, we update $p$ eigenvalues of the finite part.

This paper is organized as follows. In section 2, we give a set of solutions to the EEP-PS depending on several parametric matrices, and propose an algorithm to compute the solution. Numerical examples are given in section 3 to illustrate the performance of the proposed algorithm. Some conclusions are finally given in section 4.
\section{Eigenvalue embedding with no-spillover}
Throughout this paper, we assume that $M_u$ and $K_{\phi}$ are nonsingular. Then, it follows from the following Lemma \ref{lem-1} that $P(\lambda)$ is a regular matrix polynomial, even though the coefficient matrix $M$ of (\ref{gs-3}) is singular.
\begin{lem}\label{lem-1}
If $K_{\phi}$ is nonsingular, then $\lambda M+K$ is regular, i.e., $\det(\lambda M+K)$ is not identically zero for all $\lambda\in\mathbb{C}$.
\end{lem}

\begin{proof}
It follows from $K_{\phi}$ being nonsingular that
$$\begin{array}{rl}
\det(\lambda M+K)= & \det\left(\lambda\begin{bmatrix}
M_u & 0\\
0 & 0\\
\end{bmatrix}+\begin{bmatrix}
K_u & K_{u\phi}\\
K_{u\phi}^T & K_{\phi}\\
\end{bmatrix}\right)\\
=& \det\left(\begin{array}{cc}
\lambda M_u+K_u & K_{u\phi}\\
K_{u\phi}^T & K_{\phi}\\
\end{array}\right)\\
=& \det\left(\begin{array}{cc}
\lambda M_u+K_u-K_{u\phi}K_{\phi}^{-1}K_{u\phi}^T & 0\\
K_{u\phi}^T & K_{\phi}\\
\end{array}\right)\\
=&\det(K_{\phi})\det(\lambda M_u+K_u-K_{u\phi}K_{\phi}^{-1}K_{u\phi}^T).\\
\end{array}$$
Since $M_u$ is nonsingular, $\det(\lambda M_u+K_u-K_{u\phi}K_{\phi}^{-1}K_{u\phi}^T)$ is not identically zero for all $\lambda\in\mathbb{C}$, which implies that $\lambda M+K$ is regular.
\end{proof}

Clearly, $P(\lambda)$ has both finite and infinity eigenvalues, since the lead coefficient matrix $M$ is singular. For simplicity, we assume that the finite eigenvalues of $P(\lambda)$ are all simple and nonzero.
Taking a Jordan pair $(X(\lambda_i), J(\lambda_i))$ for every eigenvalue $\lambda_j$ of $P(\lambda)$, we define \emph{finite Jordan pair} $(X_F, J_F)$ \cite{Lancaster-book} of $P(\lambda)$ as
$$X_F=[X(\lambda_1), X(\lambda_2), \ldots, X(\lambda_{q})],\ \ J_F=\mbox{diag}(\lambda_1, \lambda_2, \ldots, \lambda_{q}).$$
From Lemma \ref{lem-1}, we known that $\deg(\det P(\lambda))=n_u$, which implies that $q=u_u$ \cite{Lancaster-book} and the sizes of $X_F$ and $J_F$ are $n\times n_u$ and $n_u\times n_u$, respectively.

\begin{lem}\label{lem-2}\cite{Lancaster-book}
Let $(\hat{X}, \hat{J})$ be a pair of matrices, where $\hat{X}$ is an $n\times n_u$ matrix and $\hat{J}$ is an nonsingular $n_u\times n_u$ diagonal matrix. Then $(\hat{X}, \hat{J})$ be a Jordan pair of $P(\lambda)=\lambda M+K$ if and only if $\mbox{rank}(\hat{X})=n_u$ and
$$M\hat{X}\hat{J}+K\hat{X}=0.$$
\end{lem}

Let
\begin{equation}\label{gs-4}
\varphi_0^{(i)},\ldots,\varphi_{s_i-1}^{(i)},\ \ i=1,\ldots, r,
\end{equation}
be a canonical set of Jordan chains of the analytic (at infinity) matrix function $\lambda^{-1}P(\lambda)$ corresponding to $\lambda=\infty$, if this point is an eigenvalue of $\lambda^{-1}P(\lambda)$. By definition, a canonical set of Jordan chains of an analytic matrix-valued function $L(\lambda)$ at infinity is just a canonical set of Jordan chains of the matrix function $L(\lambda^{-1})$ at zero \cite{Lancaster-book}. Thus, the Jordan chain (\ref{gs-4}) form a canonical set of Jordan chains of the matrix polynomial $\tilde{P}(\lambda)=\lambda P(\lambda^{-1})$ corresponding to the eigenvalue zero. Since the matrix $M_u$ in (\ref{gs-2}) is nonsingular, the algebraic multiplicity of zero eigenvalue of $\tilde{P}(\lambda)$ is $n_{\phi}$. For simplicity, we assume that the algebraic multiplicity of every Jordan block with zero eigenvalue of $\tilde{P}(\lambda)$ is $1$. Then, we shall use the following notation
$$X_{\infty}=[\varphi_1, \ldots, \varphi_{n_{\phi}}], \ \ J_{\infty}=0_{n_{\phi}\times n_{\phi}}.$$
We call $(X_{\infty}, J_{\infty})$ an \emph{infinite Jordan pair} of $P(\lambda)$.

\begin{lem}\label{lem-3}\cite{Lancaster-book}
Let $(\hat{X}, \hat{J})$ be a pair of matrices, where $\hat{X}$ is an $n\times n_{\phi}$ matrix and $\hat{J}$ is an $n_{\phi}\times n_{\phi}$ zero matrix. Then $(\hat{X}, \hat{J})$ be an infinite Jordan pair of $P(\lambda)=\lambda M+K$ if and only if $\mbox{rank}(\hat{X})=n_{\phi}$ and
$$M\hat{X}=0.$$
\end{lem}

By Lemma \ref{lem-2} and Lemma \ref{lem-3}, we have
\begin{lem}\label{lem-4}
Let $(\hat{X}, \hat{J})$ be a pair of matrices, where $\hat{X}$ is an $n\times n$ matrix and $\hat{J}=\mbox{diag}(\hat{J}_1, 0)$ with $\hat{J}_1$ being a nonsingular $n_u\times n_u$ diagonal matrix. Then $(\hat{X}, \hat{J})$ is a Jordan pair of $P(\lambda)=\lambda M+K$ if and only if $\mbox{rank}(\hat{X})=n$ and
$$M\hat{X}+K\hat{X}\begin{bmatrix}
\hat{J}_1^{-1} & 0\\
0 & 0\\
\end{bmatrix}=0.$$
And in this case, the matrix $\hat{X}$ must be of the following form
\begin{equation}\label{gg-1}
\hat{X}=\begin{bmatrix}
\hat{X}_{11} & 0\\
\hat{X}_{21} & \hat{X}_{22}\\
\end{bmatrix},
\end{equation}
where $\hat{X}_{11}\in\mathbb{R}^{n_u\times n_u}$, $\mbox{rank}(X_{11})=n_u$ and $\mbox{rank}(X_{22})=n-n_u$.
\end{lem}
\begin{proof}
The first part of this Lemma can be directly obtained by Lemma \ref{lem-2} and Lemma \ref{lem-3}. So we just prove the second part, i.e., $\hat{X}$ has the form (\ref{gg-1}). Partition $\hat{X}$ into
$$\hat{X}=[\hat{X}_1\ \ \hat{X}_2]=\begin{bmatrix}
\hat{X}_{11} & \hat{X}_{12}\\
\hat{X}_{21} & \hat{X}_{22}\\
\end{bmatrix},$$
where $\hat{X}_1\in\mathbb{R}^{n\times n_u}$, $\hat{X}_{11}\in\mathbb{R}^{n_u\times n_u}$. Since $(\hat{X}, \hat{J})$ is Jordan pair of $P(\lambda)$, we can see from the block structure of $\hat{J}$ that $(\hat{X}_2, 0_{n_{\phi}\times n_{\phi}})$ is an infinite Jordan pair of $P(\lambda)$. From Lemma \ref{lem-3}, it follows that
$$M\hat{X}_2=0,$$
which implies that $\hat{X}_{12}=0$ since $M_u$ is nonsingular. Therefore, $\hat{X}$ has the form  (\ref{gg-1}) and then $\hat{X}_{11}$ is nonsingular since $\mbox{rank}(\hat{X})=n$.
\end{proof}

For simplicity, we assume that $\{\tilde{\lambda}\}_{i=1}^p$ are simple and nonzero eigenvalues of the updated system. Without loss of generality, we assume that the $p$ eigenpairs $\{(\lambda_i, x_i)\}_{i=1}^p$ are ordered such that for $j=1,\ldots, s,$
$$\begin{array}{ll}
\lambda_{2j-1}=\overline{\lambda}_{2j}=\alpha_j+i\beta_j, & \alpha_j\in\mathbb{R}, \beta_j>0,\\
x_{2j-1}=\overline{x}_{2j}=x_{jR}+ix_{jI}, & x_{jR}, x_{jI}\in\mathbb{R}^{n},\\
\end{array}$$
and for $j=2s+1,\ldots, p$, $\lambda_j\in\mathbb{R}, x_j\in\mathbb{R}^n$. Define
\begin{equation}\label{gs-5}
\Lambda_1=\mbox{diag}\left(\begin{bmatrix}
\alpha_1 & \beta_1\\
-\beta_1 & \alpha_1\\
\end{bmatrix},\ldots, \begin{bmatrix}
\alpha_s & \beta_s\\
-\beta_s& \alpha_s\\
\end{bmatrix}, \lambda_{2s+1},\ldots, \lambda_p\right),
\end{equation}
\begin{equation}\label{gs-6}
X_1=[x_{1R}, x_{1I},\ldots, x_{sR}, x_{sI}, x_{2s+1},\ldots, x_p],
\end{equation}
which is referred to as the real representations of $\{(\lambda_i, x_i)\}_{i=1}^p$ \cite{Kang-2018-apm,Kang-amc-2014,Kang-2015-COAM}. Similarly, assume that the real representation of $\{\tilde{\lambda}_i\}_{i=1}^p$ is
\begin{equation}\label{gs-7}
\tilde{\Lambda}_1=\mbox{diag}\left(\begin{bmatrix}
\tilde{\alpha}_1 & \tilde{\beta}_1\\
-\tilde{\beta}_1 & \tilde{\alpha}_1\\
\end{bmatrix},\ldots, \begin{bmatrix}
\tilde{\alpha}_{\tilde{s}} & \tilde{\beta}_{\tilde{s}}\\
-\tilde{\beta}_{\tilde{s}}& \tilde{\alpha}_{\tilde{s}}\\
\end{bmatrix}, \tilde{\lambda}_{2{\tilde{s}}+1},\ldots, \tilde{\lambda}_p\right),
\end{equation}
and those of remaining $n-p$ eigenpairs $\{\lambda_i, x_i\}_{i=p+1}^n$ are $\Lambda_2$ and $X_2$, where
$\Lambda_2=\mbox{diag}(\Lambda_3, 0)\in\mathbb{R}^{(n-p)\times (n-p)},$
\begin{equation}\label{gs-lambda3}
\Lambda_3=\mbox{diag}\left(\begin{bmatrix}
\alpha'_1 & \beta'_1\\
-\beta'_1 & \alpha'_1\\
\end{bmatrix},\ldots, \begin{bmatrix}
\alpha'_t & \beta'_t\\
-\beta'_t& \alpha'_t\\
\end{bmatrix}, \lambda_{2t+1},\ldots, \lambda_{n_u-p}\right).
\end{equation}
We further assume that the sets of eigenvalues in $\Lambda_1$ and $\Lambda_2$ are disjoint, i.e.,
\begin{equation}\label{g-22}
\lambda(\Lambda_1)\cap\lambda(\Lambda_2)=\emptyset,
\end{equation}
where $\lambda(\cdot)$ is the set of all eigenvalues of a matrix.

With notations above, the MUP-PS can be mathematically reformulated as: given $M_u\in\mathbb{R}^{n_u\times n_u}$ and $K\in\mathbb{R}^{n\times n}$ being symmetric, $\Lambda_1, \widetilde{\Lambda}_1\in\mathbb{R}^{p\times p}$, $X_1\in\mathbb{R}^{n\times p}$  which satisfy
\begin{equation}\label{g-16}
\begin{bmatrix}
M_u & 0\\
0 & 0\\
\end{bmatrix}[X_1\ \ X_2]+K[X_1\ \ X_2]\begin{bmatrix}
\Lambda_1^{-1} & 0 & 0\\
0 & \Lambda_3^{-1} & 0\\
0 & 0 & 0\\
\end{bmatrix}=0,
\end{equation}
find $\widetilde{M}_u\in\mathbb{R}^{n_u\times n_u}$ and $\widetilde{K}\in\mathbb{R}^{n\times n}$ being symmetric such that
\begin{equation}\label{g-17}
\begin{bmatrix}
\widetilde{M}_u & 0\\
0 & 0\\
\end{bmatrix}[\widetilde{X}_1\ \ X_2]+\widetilde{K}[\widetilde{X}_1\ \ X_2]\begin{bmatrix}
\widetilde{\Lambda}_1^{-1} & 0 & 0\\
0 & \Lambda_3^{-1} & 0\\
0 & 0 & 0\\
\end{bmatrix}=0,
\end{equation}
holds for some $\widetilde{X}_1=X_1\Theta$, where $\Theta\in\mathbb{R}^{p\times p}$ is nonsingular. Note that $\Lambda_2\in\mathbb{R}^{(n-p)\times (n-p)}$ and $X_2\in\mathbb{R}^{n\times (n-p)}$ are remaining  eigenvalues and eigenvectors of the original undamped piezoelectric  system to be kept unchanged after the updating and they are generally unknown in applications. So it is necessary to  characterize the solution to the EEP-PS without the information of $\Lambda_2$ and $X_2$.

To give the parametric solutions to the EEP-PS, we will need a spectral decomposition of the undamped piezoelectric smart system of (\ref{gs-3}) which is motivated by \cite{Qiang-2017,Chu-2009-spectral}. We give the sufficient and necessary conditions for the matrices $M$ and $K$ in (\ref{gs-2}) such that  (\ref{gs-3}) is satisfied and express $M$ and $K$ in terms of eigenvalues and eigenvectors of (\ref{gs-3}). The following theorem then gives the spectral decomposition of the undamped piezoelectric  system (\ref{gs-3}), characterizing the structure of the Jordan pair $(X, J)$ and the relationship between the coefficient matrices $M$, $K$ in (\ref{gs-2}) and $(X, J)$.

\begin{thm}\label{thm-1}
Given nonsingular matrices  $X\in\mathbb{R}^{n\times n}$ and $J_1\in\mathbb{R}^{n_u\times n_u}$. Let $J=\mbox{diag}(J_1, 0)\in\mathbb{R}^{n\times n}$ and  partition $X$ as
\begin{equation}\label{g-1}
X=\begin{bmatrix}
X_u\\
X_{\phi}\\
\end{bmatrix},\ \ X_u\in \mathbb{R}^{n_u\times n},\ X_{\phi}\in\mathbb{R}^{n_{\phi}\times n}.
\end{equation}
Then there exist nonsingular symmetry matrices $M_u$ and $K$ such that
 \begin{equation}\label{g-2}
MX+KXJ=0
\end{equation}
where $M=\mbox{diag}(M_u, 0)$, if and only if there exist  symmetric matrix $\Gamma=\mbox{diag}(\Gamma_{11}, 0)\in\mathbb{R}^{n\times n}$ and nonsingular symmetric matrix $\Phi\in\mathbb{R}^{n_{\phi}\times n_{\phi}}$ , where $\Gamma_{11}\in\mathbb{R}^{n_u\times n_u}$ is nonsingular, such that
\begin{subequations}\label{g-3}
\begin{equation}\label{g-3-4}
T^{-1}:=\Gamma+X_{\phi}^T\Phi X_{\phi}
\end{equation} is nonsingular and
\begin{equation}\label{g-3-1}
J_1^T\Gamma_{11}=\Gamma_{11}J_1,
\end{equation}
\begin{equation}\label{g-3-2}
X_uTX_{\phi}^T=0,
\end{equation}
\begin{equation}\label{g-3-3}
X_{\phi}TX_{\phi}^T=\Phi^{-1}.
\end{equation}
\end{subequations}
If such matrices $\Gamma$ and $\Phi$ exist, then the coefficient matrices $M_u$ and $K$ can be expressed as
\begin{subequations}\label{g-4}
\begin{equation}\label{g-4-1}
M_u=(X_uTX_u^T)^{-1},
\end{equation}
\begin{equation}\label{g-4-2}
K=X^{-T}\begin{bmatrix}
-\Gamma_{11}J_1^{-1} & 0\\
0 & K_{22}'\\
\end{bmatrix}X^{-1},
\end{equation}
where $K'_{22}\in\mathbb{R}^{n_{\phi}\times n_{\phi}}$ is an arbitrary nonsingular symmetric matrix.
\end{subequations}

\end{thm}

\begin{proof}
(Necessity)
Pre-multiplying $X^T$ on  (\ref{g-2}) we get
\begin{equation}\label{g-6}
X^TMX+X^TKXJ=0,
\end{equation}
Substituting $M$ in (\ref{gs-2}) and $X$ in (\ref{g-1}) into (\ref{g-6}) gives
\begin{equation}\label{g-7}
X_u^TM_uX_u+X^TKXJ=0,
\end{equation}
We define $\Gamma=X_u^TM_uX_u$.
Let $\Phi\in\mathbb{R}^{n_{\phi}\times n_{\phi}}$ be a  nonsingular and symmetric matrix. Since $X$ and $M_u$ are nonsingular, it is easy to see that
\begin{equation}\label{g-8}
\Gamma+X_{\phi}^T\Phi X_{\phi}=[X_u^T\ \ X_{\phi}^T]\begin{bmatrix}
M_u & 0\\
0 & \Phi\\
\end{bmatrix}\begin{bmatrix}
X_u\\
X_{\phi}\\
\end{bmatrix}
\end{equation}
is also nonsingular. Let $T:=(\Gamma+X_{\phi}^T\Phi X_{\phi})^{-1}$. From (\ref{g-8}) it follows that
$$
\begin{bmatrix}
M_u & 0\\
0 & \Phi\\
\end{bmatrix}=\left(\begin{bmatrix}
X_u\\
X_{\phi}\\
\end{bmatrix}T[X_u^T\ \ X_{\phi}^T]\right)^{-1}=\left(\begin{bmatrix}
X_uTX_u^T & X_uTX_{\phi}^T\\
X_{\phi}TX_u^T & X_{\phi}TX_{\phi}^T\\
\end{bmatrix}\right)^{-1},$$
which implies that $X_uTX_{\phi}^T=0$, $X_{\phi}TX_{\phi}^T=\Phi^{-1}$ and $M_u=(X_uTX_u^T)^{-1}$, which are exactly the formulas as in (\ref{g-3-2}), (\ref{g-3-3}) and (\ref{g-4-1}), respectively.

Substituting $J=\mbox{diag}(J_1, 0)$ into (\ref{g-7}) , we have
\begin{equation}\label{g-5}
\Gamma+X^TKX\begin{bmatrix}
J_1 & 0\\
0 & 0\\
\end{bmatrix}=0.
\end{equation}
Let $$\Gamma=\begin{bmatrix}
\Gamma_{11} & \Gamma_{12}\\
\Gamma_{12}^T & \Gamma_{22}\\
\end{bmatrix}, \ \ X^TKX=\begin{bmatrix}
K_{11} & K_{12}\\
K_{12}^T & K_{22}\\
\end{bmatrix},$$
where $\Gamma_{11}, K_{11}\in\mathbb{R}^{n_u\times n_u}$. Clearly, (\ref{g-5}) holds if and only if
\begin{equation}\label{g-11}
\Gamma_{12}=0, \ \ \Gamma_{22}=0,
\end{equation}
and in this case, we have
\begin{equation}\label{g-12}
K_{11}=-\Gamma_{11}J_1^{-1},\ \ K_{12}=0.
\end{equation}
Since $K_{11}$ is symmetric, we must have
\begin{equation}\label{g-13}
J_1^T\Gamma_{11}=\Gamma_{11}J_1.
\end{equation}
Therefore, we have
\begin{equation}\label{g-14}
K=X^{-T}\begin{bmatrix}
-\Gamma_{11}J_1^{-1} & 0\\
0 & K_{22}'\\
\end{bmatrix}X^{-1},
\end{equation}
where $K_{22}\in\mathbb{R}^{n_{\phi}\times n_{\phi}}$ is an arbitrary nonsingular and symmetric matrix.
From (\ref{g-13}), we can see that $K$ is nonsingular and $K^T=K$.

(Sufficiency) Conversely, if there exist two symmetric matrices $\Gamma$ and $\Phi$ satisfying (\ref{g-3}), where $\Phi$ is nonsingular, then
\begin{equation}\label{g-15}
XTX^T=\begin{bmatrix}
X_u\\
X_{\phi}\\
\end{bmatrix}T[X_u^T\ \ X_{\phi}^T]=\begin{bmatrix}
X_uTX_u^T & 0\\
0 & \Phi^{-1}\\
\end{bmatrix},
\end{equation}
is also nonsingular, which means that $X_uTX_u^T$ is nonsingular and symmetric. Thus $M_u$ is well defined by (\ref{g-4-1}).
From (\ref{g-15}), it follows that
$$X^{-T}T^{-1}X^{-1}=\begin{bmatrix}
(X_uTX_u^T)^{-1} & 0\\
0 & \Phi\\
\end{bmatrix},$$
which implies that
\begin{equation}\label{g-9}
T^{-1}=X_u^T(X_uTX_u^T)^{-1}X_u+X_{\phi}^T\Phi X_{\phi}.
\end{equation}
By the definitions of $M_u$ and $T$, we can see from (\ref{g-9}) that $\Gamma=X_u^TM_uX_u$. Noting the symmetry of $K_{22}'$ and (\ref{g-3-1}), it is easy to see that $K$ defined by (\ref{g-4-2}) is symmetric. By (\ref{g-3-1}) and (\ref{g-4-2}), we have
$$\begin{array}{rl}
X^TMX+X^TKXJ& =X_u^TM_uX_u +X^TKX J\\
& =\Gamma+X^TKXJ\\
& =\begin{bmatrix}
\Gamma_{11} & 0\\
0 & 0\\
\end{bmatrix}+X^TKX\begin{bmatrix}
J_1 & 0\\
0 & 0\\
\end{bmatrix}\\
& =0,\\
\end{array}$$
i.e., $MX+KXJ=0$ since $X$ is nonsingular. This completes the proof of the theorem.
\end{proof}

Now, with Theorem \ref{thm-1}, we are ready to characterize the solutions to the EEP-PS. Let $X_1$ in (\ref{gs-6}) and $\Lambda_1$ in (\ref{gs-5}) be the eigendata which will be replaced by $\widetilde{X}_1\in\mathbb{R}^{n\times p}$ and $\widetilde{\Lambda}_1$ in (\ref{gs-7}, and those of remaining  $n-p$ eigendata be given by $X_2\in\mathbb{R}^{n\times (n-p)}$ and $\Lambda_2$. From Lemma \ref{lem-4} and Theorem \ref{thm-1}, without loss of generality, we may always assume that $[X_1, X_2]$ and $[\widetilde{X}_1, X_2]$ are nonsingular. Partition $X_1$ as
\begin{equation}\label{gs-x}
X_1=\begin{bmatrix}
X_{1u}\\
X_{1\phi}\\
\end{bmatrix}
 \end{equation} where $X_{1u}\in\mathbb{R}^{n_u\times p}$. Clearly, $\left([X_1, X_2], \mbox{diag}(\Lambda_1, \Lambda_2)\right)$ is a Jordan pair of $P(\lambda)$.The following theorem is our main result which gives the parametric solutions to the EEP-PS.

\begin{thm}\label{thm-2}
For any nonsingular matrices $\Theta\in\mathbb{R}^{p\times p}$ and $\widetilde{\Gamma}_1\in\mathbb{R}^{p\times p}$ with
\begin{equation}\label{gs-11}
\widetilde{\Gamma}_1=\mbox{diag}\left(\begin{bmatrix}
a_1 & b_1\\
b_1 & -a_1\\
\end{bmatrix},\ldots, \begin{bmatrix}
a_{\tilde{s}} & b_{\tilde{s}}\\
b_{\tilde{s}}& -a_{\tilde{s}}\\
\end{bmatrix}, c_{2{\tilde{s}}+1},\ldots, c_p\right),
\end{equation}
the matrices
\begin{equation}\label{gs-13}
\widetilde{M}_u=\left(M_u^{-1}-X_{1u}\Gamma^{-1}_1X_{1u}^T+X_{1u}\Theta\widetilde{\Gamma}^{-1}_1\Theta^TX_{1u}^T\right)^{-1},
\end{equation}
\begin{equation}\label{gs-10}
\widetilde{K}=\left(K^{-1}+X_1\Lambda_1\Gamma_1^{-1}X_1^T-X_1\Theta\widetilde{\Lambda}_1\widetilde{\Gamma}_1^{-1}\Theta^TX_1^T\right)^{-1},
\end{equation}
where
\begin{equation}\label{gs-gamma}
\Gamma_1=X_{1u}^TM_uX_{1u},
\end{equation}
form a solution to the EEP-PS, provided that $\widetilde{M}_u$ and $\widetilde{K}$ in (\ref{gs-13}), (\ref{gs-10}), respectively, are well defined.
\end{thm}

\begin{proof}
 Since $p\leq n_u$ and $[X_1, X_2]$ is nonsingular, we can see from Lemma \ref{lem-4} that $\mbox{rank}(X_1)=p$. Then, there exists a nonsingular matrix $Q:=\begin{bmatrix}
I_{n_u} & 0\\
Q' & I_{n_{\phi}}\\
\end{bmatrix}\in\mathbb{R}^{n\times n}$  such that
\begin{equation}\label{gs-q}
X_1=Q\begin{bmatrix}
X_{1u}\\
0\\
\end{bmatrix}, \ \ \  Q^T\begin{bmatrix}
M_u & 0\\
0 & 0\\
\end{bmatrix}Q=\begin{bmatrix}
M_u & 0\\
0 & 0\\
\end{bmatrix}.
\end{equation}
Partition $Q^{-1}X_2$ as
$
Q^{-1}X_2=\begin{bmatrix}
X_{2u}\\
X_{2\phi}\\
\end{bmatrix},$ where $ X_{2u}\in\mathbb{R}^{n_u\times (n-p)}
$, and then (\ref{g-16}) can be equivalently rewritten into
\begin{equation}\label{gs-19-1}
\begin{bmatrix}
M_u & 0\\
0 & 0\\
\end{bmatrix}\begin{bmatrix}
X_{1u} & X_{2u}\\
0 & X_{2\phi}\\
\end{bmatrix}+Q^{T}KQ\begin{bmatrix}
X_{1u} & X_{2u}\\
0 & X_{2\phi}\\
\end{bmatrix}\begin{bmatrix}
\Lambda_1^{-1} & 0 & 0\\
0 & \Lambda_3^{-1} & 0\\
0 & 0 & 0\\
\end{bmatrix}=0.
\end{equation}

By (\ref{gs-19-1}) and Theorem \ref{thm-1}, we first give the expressions of  the coefficient matrices $M_u$ and $Q^TKQ$ by the eigenvalue matrix $\Lambda_1$, $\Lambda_3$ and the eigenvectors matrix
$$
\begin{bmatrix}
X_{1u} & X_{2u}\\
0 & X_{2\phi}\\
\end{bmatrix}.$$
 Let $X_u=[X_{1u}\ \ X_{2u}]$, $X_{\phi}=[0_{n_{\phi}\times p}\ \ X_{2\phi}]$. Similar to the proof of Theorem \ref{thm-1}, we define
$$\Gamma=X_u^TM_uX_u, \ \ \ T^{-1}=\Gamma+X_{\phi}^T\Phi X_{\phi},$$
where $\Phi\in\mathbb{R}^{n_{\phi}\times n_{\phi}}$ is a  nonsingular and symmetric matrix.  Then
\begin{equation}\label{gs-gamma-1}
\Gamma=\begin{bmatrix}
X_{1u}^T\\
X_{2u}^T\\
\end{bmatrix}M_u\begin{bmatrix}
X_{1u} & X_{2u}\\
\end{bmatrix},
\end{equation}
and \begin{equation}\label{g-T}
T^{-1}=\Gamma+\begin{bmatrix}
0_{p\times p} & 0\\
0 & X_{2\phi}^T\Phi X_{2\phi}\\
\end{bmatrix}.
\end{equation}
Using (\ref{gs-19-1}), we can see from the proof of Theorem \ref{thm-1} that $\Gamma=\mbox{diag}(\Gamma_{11}, 0)$, where $\Gamma_{11}\in\mathbb{R}^{n_u\times n_u}$, and the matrices $T, \Gamma_{11}$ satisfy
\begin{equation}\label{g-19}
\begin{bmatrix}
\Lambda_1^{-1} & 0 \\
0 & \Lambda_3^{-1}\\
\end{bmatrix}^T\Gamma_{11}=\Gamma_{11}\begin{bmatrix}
\Lambda_1^{-1} & 0 \\
0 & \Lambda_3^{-1}\\
\end{bmatrix}
\end{equation}
\begin{equation}\label{g-20}
X_uTX_{\phi}^T=[X_{1u}\ \ X_{2u}]T\begin{bmatrix}
0\\
X_{2\phi}^T\\
\end{bmatrix}=0,
\end{equation}
\begin{equation}\label{g-21}
X_{\phi}TX_{\phi}^T=[0\ \ X_{2\phi}]T\begin{bmatrix}
0\\
X_{2\phi}^T\\
\end{bmatrix}=\Phi^{-1}.
\end{equation}
From the assumption (\ref{g-22}) and (\ref{g-19}), it follows that $\Gamma_{11}$ must be of the block diagonal form $\Gamma_{11}=\mbox{diag}(\Gamma_1, \Gamma_2)$, where $\Gamma_1\in\mathbb{R}^{p\times p}$, and
$$\Gamma_1\Lambda^{-1}_1=\Lambda_1^{-T}\Gamma_1,\ \  \Gamma_2\Lambda_3^{-1}=\Lambda_3^{-T}\Gamma_2.$$
Substituting $\Gamma=\mbox{diag}(\Gamma_1, \Gamma_2, 0)$ into (\ref{gs-gamma-1}), it follows from the comparison of two sides of the (\ref{gs-gamma-1}) that $\Gamma_1=X_{1u}^TM_uX_{1u}$, i.e., $\Gamma_1$ is of form (\ref{gs-gamma}).

By (\ref{g-T}), we can see that $T$ must be of the block diagonal form $T=\mbox{diag}(T_1, T_2)$, where $T_1\in\mathbb{R}^{p\times p}$, and
$$T_1^{-1}=\Gamma_1, \ \ T_2^{-1}=\mbox{diag}(\Gamma_2, 0_{n_{\phi}\times n_{\phi}})+X_{2\phi}^T\Phi X_{2\phi}.$$
Hence (\ref{g-20}) and (\ref{g-21}) can be equivalently rewritten as
\begin{equation}\label{g-25}
X_{2u}T_2X_{2\phi}^T=0,
\end{equation}
\begin{equation}\label{g-26}
X_{2\phi}T_2X_{2\phi}^T=\Phi^{-1}.
\end{equation}
and the matrices $M_u$ and $K$ of (\ref{gs-19-1}) can be expressed as
\begin{equation}\label{g-26-1}
M_u=(X_{1u}\Gamma_1^{-1}X_{1u}^T+X_{2u}T_2X_{2u}^T)^{-1},
\end{equation}
\begin{equation}\label{g-27}
\begin{array}{rl}
Q^TKQ&=\left(\begin{bmatrix}
X_{1u} & X_{2u}\\
0 & X_{2\phi}\\
\end{bmatrix}\begin{bmatrix}
-\Lambda_1\Gamma_1^{-1} & 0\\
0 & \widehat{K}_{22}\\
\end{bmatrix}\begin{bmatrix}
X_{1u}^T & 0\\
X_{2u}^T & X_{2\phi}\\
\end{bmatrix}\right)^{-1}\\
&=\left(Q^{-1}\begin{bmatrix}
X_{1} & X_{2}\\
\end{bmatrix}\begin{bmatrix}
-\Lambda_1\Gamma_1^{-1} & 0\\
0 & \widehat{K}_{22}\\
\end{bmatrix}\begin{bmatrix}
X_{1}^T \\
X_{2}^T\\
\end{bmatrix}Q^{-T}\right)^{-1},\\
\end{array}
\end{equation}
where $\widehat{K}_{22}=\begin{bmatrix}
-\Lambda_3\Gamma_{2}^{-1} & 0\\
0 & K_{22}'\\
\end{bmatrix}$, $K_{22}'$ is an arbitrary nonsingular $n_{\phi}\times n_{\phi}$ symmetry matrix,
which implies that the coefficient matrices of original undamped    system can be expressed as
\begin{equation}\label{g-26-2}
M=\begin{bmatrix}
(X_{1u}\Gamma_1^{-1}X_{1u}^T+X_{2u}T_2X_{2u}^T)^{-1} & 0\\
0 & 0\\
\end{bmatrix},
\end{equation}
\begin{equation}\label{g-27-1}
K=\left(\begin{bmatrix}
X_{1} & X_{2}\\
\end{bmatrix}\begin{bmatrix}
-\Lambda_1\Gamma_1^{-1} & 0\\
0 & \widehat{K}_{22}\\
\end{bmatrix}\begin{bmatrix}
X_{1}^T \\
X_{2}^T\\
\end{bmatrix}\right)^{-1}.
\end{equation}

Let $\widetilde{X}_1\in\mathbb{R}^{n\times p}$ and $[\widetilde{X}_1\ \ X_2]$ is nonsingular. Since $Q$ is nonsingular and satisfies (\ref{gs-q}),
it is easy to verify that (\ref{g-17}) is equivalent to
\begin{equation}\label{g-17-1}
\begin{bmatrix}
\widetilde{M}_u & 0\\
0 & 0\\
\end{bmatrix}[Q^{-1}\widetilde{X}_1\ \ Q^{-1}X_2]+Q^T\widetilde{K}Q[Q^{-1}\widetilde{X}_1\ \ Q^{-1}X_2]\begin{bmatrix}
\widetilde{\Lambda}_1^{-1} & 0 & 0\\
0 & \Lambda_3^{-1} & 0\\
0 & 0 & 0\\
\end{bmatrix}=0,
\end{equation}
Next, for given matrices $\widetilde{X}_1$ and $\widetilde{\Lambda}_1$, we will find $\widetilde{M}_u$ and $\widetilde{K}$ such that (\ref{g-17-1}) is satisfied.
Let $Q^{-1}\widetilde{X}_1=\begin{bmatrix}
\widetilde{X}_{1u}\\
\widetilde{X}_{1\phi}\\
\end{bmatrix}$ with $\widetilde{X}_{1u}\in\mathbb{R}^{n_u\times p}$. Recall that $Q^{-1}X_2=\begin{bmatrix}
 X_{2u}\\
 X_{2\phi}\\
 \end{bmatrix}$, then $\begin{bmatrix}
 \widetilde{X}_u\\
 \widetilde{X}_{\phi}\\
 \end{bmatrix}:=[Q^{-1}\widetilde{X}_1\ \ Q^{-1}X_2]$ is nonsingular, where
 $$\widetilde{X}_u=[\widetilde{X}_{1u}\ \ X_{2u}],\ \ \widetilde{X}_{\phi}=[\widetilde{X}_{1\phi}\ \ X_{2\phi}].$$
 From Theorem \ref{thm-1} we know that if there exist nonsingular and symmetric matrices $\widetilde{\Gamma}_{11}=\mbox{diag}(\widetilde{\Gamma}_1, \widetilde{\Gamma}_2)$, and $\widetilde{\Phi}\in\mathbb{R}^{n_{\phi}\times n_{\phi}}$, where $\widetilde{\Gamma}_1\in\mathbb{R}^{p\times p}$ and $\widetilde{\Gamma}_2\in\mathbb{R}^{(n_u-p)\times (n_u-p)}$ satisfying
\begin{equation}\label{gs-14}
\widetilde{\Gamma}_1\widetilde{\Lambda}^{-1}_1=\widetilde{\Lambda}_1^{-T}\widetilde{\Gamma}_1,\ \  \widetilde{\Gamma_2}\Lambda_3^{-1}=\Lambda_3^{-T}\widetilde{\Gamma_2},
\end{equation}
such that
\begin{equation}\label{g-32}
\widetilde{X}_u\widetilde{T}\widetilde{X_{\phi}}^T=0,
\end{equation}
\begin{equation}\label{g-33}
\widetilde{X}_{\phi}\widetilde{T}\widetilde{X}_{\phi}^T=\widetilde{\Phi}^{-1},
\end{equation}
where
\begin{equation}\label{gs-T}
\widetilde{T}^{-1}=\begin{bmatrix}
\widetilde{\Gamma}_{11} & 0\\
0 & 0_{n_{\phi}\times n_{\phi}}\\
\end{bmatrix}+\widetilde{X}_{\phi}^T\widetilde{\Phi}\widetilde{X}_{\phi},
\end{equation}
then the matrices $\widetilde{M}_u$ and $\widetilde{K}$ defined by
\begin{equation}\label{g-30}
\widetilde{M}_u=(\widetilde{X}_{u}\widetilde{T}\widetilde{X}_{u}^T)^{-1},
\end{equation}
\begin{equation}\label{g-31}
Q^T\widetilde{K}Q=\left(\begin{bmatrix}
\widetilde{X}_u\\
\widetilde{X}_{\phi}\\
\end{bmatrix}\begin{bmatrix}
-\widetilde{\Lambda}_1\widetilde{\Gamma}_1^{-1} & 0\\
0 & \widetilde{K}_{22}\\
\end{bmatrix}\begin{bmatrix}
\widetilde{X}_{u}^T & \widetilde{X}_{\phi}^T\\
\end{bmatrix}\right)^{-1},
\end{equation}
where $\widetilde{K}_{22}=\begin{bmatrix}
-\Lambda_3\widetilde{\Gamma}_{2}^{-1} & 0\\
0 & K_{22}'\\
\end{bmatrix}\in\mathbb{R}^{(n-p)\times (n-p)}$, satisfy (\ref{g-2}), i.e., they form a solution to the EEP-PS. Similar as (\ref{g-15}) in the proof of Theorem \ref{thm-1}, we can see that if (\ref{g-32}) and (\ref{g-33}) are satisfied, the matrix $\begin{bmatrix}
\widetilde{X}_u\\
\widetilde{X}_{\phi}\\
\end{bmatrix}$ being nonsingular ensures that $\widetilde{M}_u$ and $\widetilde{K}$ in (\ref{g-30}) and (\ref{g-31}) are well defined, and vice versa.

Since $\widetilde{\Lambda}_1$ is of the form (\ref{gs-7}), any nonsingular matrix $\widetilde{\Gamma}_1$ of the form (\ref{gs-11}) is symmetric and satisfies $\widetilde{\Gamma}_1\widetilde{\Lambda}^{-1}_1=\widetilde{\Lambda}_1^{-T}\widetilde{\Gamma}_1$ in (\ref{gs-14}). Noting
\begin{equation}\label{gs-8}
\widetilde{X}_1=X_1\Theta,
\end{equation}
we have
$$\begin{bmatrix}
\widetilde{X}_{1u}\\
\widetilde{X}_{1\phi}\\
\end{bmatrix}=Q^{-1}\widetilde{X}_1=Q^{-1}X_1\Theta=\begin{bmatrix}
X_{1u}\\
0\\
\end{bmatrix}\Theta=
\begin{bmatrix}
X_{1u}\Theta\\
0\\
\end{bmatrix},$$
which implies that $\widetilde{X}_{1\phi}=0$, $\widetilde{X}_{1u}=X_{1u}\Theta$ and $\widetilde{X}_{\phi}=X_{\phi}$. And then $\widetilde{T}$ given by (\ref{gs-T}) must be of the block form
$\widetilde{T}=\mbox{diag}(\widetilde{T}_1, \widetilde{T}_2),$
where
$$\widetilde{T}^{-1}_1=\widetilde{\Gamma}_1\in\mathbb{R}^{p\times p}, \ \ \ \widetilde{T}^{-1}_2=\begin{bmatrix}
\widetilde{\Gamma}_2 & 0\\
 0 & 0_{n_{\phi}\times n_{\phi}}\\
 \end{bmatrix}+X_{2\phi}^T\widetilde{\Phi}X_{2\phi}\in\mathbb{R}^{(n-p)\times (n-p)}.$$
Substituting $\widetilde{T}$ into (\ref{g-32}) and (\ref{g-33}), we have
\begin{equation}\label{g-32-1}
\widetilde{X}_u\widetilde{T}\widetilde{X_{\phi}}^T=[\widetilde{X}_{1u}\ \ X_{2u}]\widetilde{T}\begin{bmatrix}
\widetilde{X}_{1\phi}^T\\
X_{2\phi}^T\\
\end{bmatrix}=X_{2u}\widetilde{T}_2X_{2\phi}^T=0,
\end{equation}
\begin{equation}\label{g-33-1}
\widetilde{X}_{\phi}\widetilde{T}\widetilde{X_{\phi}}^T=[\widetilde{X}_{1\phi}\ \ \widetilde{X}_{2\phi}]\widetilde{T}\begin{bmatrix}
\widetilde{X}_{1\phi}^T\\
X_{2\phi}^T\\
\end{bmatrix}=X_{2\phi}\widetilde{T}_2X_{2\phi}^T=\widetilde{\Phi}^{-1}.
\end{equation}
If we choose $\widetilde{\Phi}=\Phi$ and $\widetilde{\Gamma_2}=\Gamma_2$, then $\widetilde{T}_2=T_2$, $\widetilde{K}_{22}=\widehat{K}_{22}$.
It follows from (\ref{g-25}), (\ref{g-26}), (\ref{g-32-1}) and (\ref{g-33-1} that
$$\widetilde{X}_u\widetilde{T}\widetilde{X_{\phi}}^T=X_{2u}\widetilde{T}_2X_{2\phi}^T=X_{2u}T_2X_{2\phi}^T=0,$$
$$\widetilde{X}_{\phi}\widetilde{T}\widetilde{X_{\phi}}^T=X_{2\phi}\widetilde{T}_2X_{2\phi}^T=X_{2\phi}T_2X_{2\phi}^T=\Phi^{-1},$$
which imply that  (\ref{g-32}) and (\ref{g-33}) are satisfied. Hence, the matrices $\widetilde{M}_u$ and $\widetilde{K}$ defined by (\ref{g-30}) and (\ref{g-31}), respectively, form a solution to the EEP-PS, provided that $\widetilde{M}_u$ and $\widetilde{K}$ in (\ref{g-30}) and (\ref{g-31}) are well defined.

Substituting (\ref{gs-8}) into (\ref{g-30}) and (\ref{g-31}), it follows from (\ref{g-26-1}) and (\ref{g-27-1}) that
$$\widetilde{M}_u=\left(M_u^{-1}-X_{1u}\Gamma^{-1}_1X_{1u}^T+X_{1u}\Theta\widetilde{\Gamma}^{-1}_1\Theta^TX_{1u}^T\right)^{-1},$$
$$\begin{array}{rl}
(Q^T\widetilde{K}Q)^{-1}=&\begin{bmatrix}
Q^{-1}\widetilde{X}_1 & Q^{-1}X_2\\
\end{bmatrix}\begin{bmatrix}
-\widetilde{\Lambda}_1\widetilde{\Gamma}_1^{-1} & 0\\
0 & \widehat{K}_{22}\\
\end{bmatrix}\begin{bmatrix}
\widetilde{X}_1^TQ^{-T}\\
X_2^TQ^{-T}\\
\end{bmatrix}\\
=&Q^{-1}\begin{bmatrix}
X_1 & X_2\\
\end{bmatrix}\begin{bmatrix}
-\Theta\widetilde{\Lambda}_1\widetilde{\Gamma}_1^{-1}\Theta^T & 0\\
0 & \widehat{K}_{22}\\
\end{bmatrix}\begin{bmatrix}
X_1^T\\
X_2^T\\
\end{bmatrix}Q^{-T}\\
=& Q^{-1}\left(-X_1\Theta\widetilde{\Lambda}_1\widetilde{\Gamma}_1^{-1}\Theta^TX_1^T+X_2\widehat{K}_{22}X_2^T\right)Q^{-T}\\
=& Q^{-1}\left(K^{-1}-X_1\Theta\widetilde{\Lambda}_1\widetilde{\Gamma}_1^{-1}\Theta^TX_1^T+X_1\Lambda_1\Gamma_1^{-1}X_1^T\right)Q^{-T}.\\
\end{array}$$
It follows that
$$\widetilde{K}=\left(K^{-1}-X_1\Theta\widetilde{\Lambda}_1\widetilde{\Gamma}_1^{-1}\Theta^TX_1^T+X_1\Lambda_1\Gamma_1^{-1}X_1^T\right)^{-1},$$
which are exactly the formulas (\ref{gs-13}) and  (\ref{gs-10}).
\end{proof}

In the formula (\ref{gs-10}) for $\widetilde{K}$, inverses of $n\times n$ matrices is involved. If $p$ is less than $n_u$ and $n_{\phi}$, this formula can be reformulated as, by using the Sherman-Morrison-Woodbury formula \cite{Golub-book-1996}
\begin{equation}\label{ggg-1}
\widetilde{M}_u=M_u-M_u(X_{1u}\Theta\widetilde{\Gamma}^{-1}_1\Theta^T-X_{1u}\Gamma^{-1}_1)\left(I+X_{1u}^TM_u(X_{1u}\Theta\widetilde{\Gamma}^{-1}_1\Theta^T-X_{1u}\Gamma^{-1}_1)\right)^{-1}X_{1u}^TM_u,
\end{equation}
\begin{equation}\label{gs-13-1}
\widetilde{K}=K-K(X_1\Lambda_1\Gamma_1^{-1}-X_1\Theta\widetilde{\Lambda}_1\widetilde{\Gamma}_1^{-1}\Theta^T)\left(I+X_1^TK(X_1\Lambda_1\Gamma_1^{-1}-X_1\Theta\widetilde{\Lambda}_1\widetilde{\Gamma}_1^{-1}\Theta^T)\right)^{-1}X_1^TK,
\end{equation}
where only inverses of $p\times p$ matrices are needed.

Under the assumption that $\{\lambda_j\}_{j=1}^p$ are nonzero simple eigenvalues, the corresponding $\Gamma_1$ must be of the form similar to (\ref{gs-11}) with $\tilde{s}$ replaced by $s$, since $\Gamma_1$ must be symmetric and satisfy (\ref{g-3-1}). Similarly, $\widetilde{\Gamma}_1$ of the form (\ref{gs-11}) also follows from the assumption that $\{\tilde{\lambda}\}_{j=1}^p$ are nonzero simple eigenvalues. We should point out that the numbers of complex conjugate eigenvalues and real eigenvalues of the updated system may not be same as those of the original system, since the conditions that  the matrix $\widetilde{\Gamma}_1$ should satisfy is $\widetilde{\Gamma}^T=\widetilde{\Gamma}$  and (\ref{g-3-1}).

From Theorem \ref{thm-2}, we known that if all the eigenvalues of $\widetilde{\Lambda}_1$ are simple and nonzero, and $[\widetilde{X}_1\ \ X_2]$ is nonsingular, then EEP-PS is  solvable, and for any nonsingular $\Theta$ and $\widetilde{\Gamma}_1$ we can give a  parametric solution to the EEP-PS. A trivial way to choose $\Theta$ and $\widetilde{\Gamma}_1$ is $\Theta=I$ and $\widetilde{\Gamma}_1=\Gamma_1$. In this case $\widetilde{X}_1=X_1\Theta=X_1$, i.e., the eigenvectors in $X_1$ are also kept unchanged in the updating. And $\widetilde{\Gamma}_1=\Gamma_1$ generally restricts $\tilde{s}=s$, i.e., the numbers of complex conjugate eigenvalues and real eigenvalues of the updated system must be same as those of the original system.

In all, the following Algorithm 1 can be used to find a solution to the EEP-PS.

\begin{algorithm}[h]\label{alg-2}
\caption{Find solution to EEP-PS.} 
\hspace*{0.02in} {\bf Input:} 
$X_1$, $\Lambda_1$, $M_u$, $K$.\\
\hspace*{0.02in} {\bf Output:} 
$\widetilde{M}_u$ and $\widetilde{K}$.
\begin{algorithmic}[1]
\State Partition $X_1$ as in (\ref{gs-x}). 
\State Choose $\Theta$ and  $\widetilde{\Gamma}_1$ in (\ref{gs-11}).
\State Compute $\Gamma_1$ as in (\ref{gs-gamma}).
\State Compute $\widetilde{M}_u$ and $\widetilde{K}$ as in (\ref{gs-13}) and (\ref{gs-10}), respectively.
\end{algorithmic}
\end{algorithm}

It is worthwhile to point out that Algorithm 1 does not need any information of $\Lambda_2$ and $X_2$, which are  the remaining $n-p$ eigenvalues and eigenvectors to be kept unchanged. We can also see from Algorithm 1 that there are many freedoms in choosing $\Theta$ and $\widetilde{\Gamma}_1$, which can be further exploited to achieve some other desirable properties. For example, we can wish the updated matrices are approximate to the coefficient matrices, which is to minimize
\begin{equation}\label{gg-r}
Rec.MK=\tau_1\frac{||M_u-\widetilde{M}_u||_2}{||M_u||_2}+\tau_2\frac{||K-\widetilde{K}||_2}{||K||_2},
\end{equation}
where $\tau_1, \tau_2$ are weight factors to balance all terms and $\tau_1>0, \tau_2>0$. This is a constrained optimization problem, since $\Theta$ is requires to be nonsingular. If we choose certain nonsingular matrix $\Theta$, and leave $\widetilde{\Gamma}_1$ be the free parameter matrix, then it will become a simple unconstrained optimization problem, which can be solved by the MATLAB function {\bf fminunc}.

In Step 4. $\widetilde{M}_u$ can be computed by either (\ref{gs-13}) or (\ref{ggg-1}), and $\widetilde{K}$ can be computed by either (\ref{gs-10}) or (\ref{gs-13-1}).
Note that (\ref{gs-13}) and (\ref{gs-10}) require inverses of matrices of order $n_u$ and $n$, respectively, while (\ref{ggg-1}) and (\ref{gs-13-1}) require inverses of matrices of order $p$. However, if $n$ and $n_u$ is much greater than $p$, (\ref{ggg-1}) and (\ref{gs-13-1}) not only costs less than (\ref{gs-13}) and (\ref{gs-10}), respectively, but also generally leads to more accurate solutions. And we will provide some numerical examples to illustrate it in the next section.

\section{Numerical examples}

In order to illustrate the performance of Algorithm 1, we present some numerical examples. All computations were carried out in MATLAB 2017a with machine epsilon $\epsilon\approx 2.2\times 10^{-16}$.  In these examples we compute the relative residuals of the updated system (Res.U) and the original system (Res.O) as
$$\mbox{Res1.U}=\frac{||\widetilde{M}\widetilde{X}_1\widetilde{\Lambda}_1+\widetilde{K}\widetilde{X}_1||_2}{(||\widetilde{M}||_2||\widetilde{\Lambda}_1||_2+||\widetilde{K}||_2)||\widetilde{X}_1||_2},\ \ \ \mbox{Res2.U}=\frac{||\widetilde{M}X_2+\widetilde{K}X_2\Lambda_2'||_2}{(||\widetilde{M}||_2||+||\widetilde{K}||_2||\Lambda_2'||_2)||\widetilde{X}_1||_2},$$
$$\mbox{Res1.O}=\frac{||MX_1\Lambda_1+KX_1||_2}{(||M||_2||\Lambda_1||_2+||K||_2)||X_1||_2},\ \ \ \mbox{Res2.O}=\frac{||MX_2+KX_2\Lambda_2'||_2}{(||M||_2||+||K||_2||\Lambda_2'||_2)||X_2||_2},$$
where $\Lambda_2'=\mbox{diag}(\Lambda_3^{-1}, 0)\in\mathbb{R}^{(n-p)\times (n-p)},$ and $\Lambda_3$ is given by  (\ref{gs-lambda3}) which is the real representation of eigenvalues to be kept unchanged. For an randomly generated nonsingular matrix $\Theta$, results obtained by Algorithm 1 with two choices of $\widetilde{\Gamma}_1$ in Step 2:

{\rm (a)} choose $\widetilde{\Gamma}_1=\Gamma_1$;

{\rm (b)} take $\widetilde{\Gamma}_1$ as a free parameter matrix;
\noindent are respectively denoted by `$*_{-}$a', `$*_{-}$b'.

\begin{exa}\label{exa-1}
In this example, the matrices $M, K$ of the form (\ref{gs-2}) in the original undamped piezoelectric smart system are randomly generated with $n_u=100$ and $n_{\phi}=40$. Suppose that we are to update the following six eigenvalues
$$\{\lambda_j\}_{j=1}^6=\{0.1117\pm 0.8733i, 0.2685\pm 0.2672i, 0.6055, 0.3568\}$$
to another six  eigenvalues $\{\tilde{\lambda}_i\}_{i=1}^6$, including three real eigenvalues and one pair of complex conjugate eigenvalues, with random perturbations $\Delta \lambda_j=\lambda_j-\tilde{\lambda}_j$, where $|\Delta \lambda_j|\leq 0.3$, and keep the remaining eigenvalues and corresponding eigenvectors unchanged. In this example, those eigenvalues $\tilde{\lambda}$ are
$$\{\tilde{\lambda}_j\}_{j=1}^6=\{0.0645\pm 0.6315i, 0.4044\pm 0.4044i, 0.5463, 0.0041\}.$$
\end{exa}

In this example, we updated the system by two choices of $\widetilde{\Gamma}_1$, and the relative residues are
$$\mbox{Res1.U.a}=1.2434e-14,\ \ \ \ \mbox{Res2.U.a}=3.2196e-16,$$
$$\mbox{Res1.U.b}=1.4777e-14,\ \ \ \ \mbox{Res2.U.b}=1.6653e-16,$$
which are comparable with those of the original system
$$\mbox{Res1.O}=1.3402e-16,\ \ \ \ \mbox{Res2.O}=1.1102e-16.$$

For choice {\rm (b)}, we choose certain nonsingular matrix $\Theta$ and take $\widetilde{\Gamma}_1$ with form (\ref{gs-11}) as the free parameter matrix. Then we use the MATLAB function {\bf fminunc} to compute a solution to the EEP-PS such that the distance between the original system and updated systems, the {\rm Rec.MK} defined as in (\ref{gg-r}) with $\tau_1=\tau_2=1$, is minimized. The minimum {\rm Rec.MK} obtained by {\bf fminunc} is $0.0.2382$, which is smaller than $\mbox{Rec.MK}=8.6394$ with choice {\rm (a)}, which implies that exploiting the freedoms of parametric matrix $\widetilde{\Gamma}_1$ does lead to smaller updates on the coefficient matrices.

\begin{exa}\label{exa-2}
Let $M, K$ are the same as in Example \ref{exa-1}. In this example, we are to update the six eigenvalues $\lambda_1, \ldots, \lambda_6$ in Example \ref{exa-1} to one pair of complex conjugate eigenvalues and four real ones, with similar random perturbations. The new eigenvalues $\tilde{\lambda}_j$ are
$$\{\tilde{\lambda}_j\}_{j=1}^6=\{0.8480\pm 0.5641i, 0.1445, 0.6530,0.5392, 0.7539\}.$$
\end{exa}

Since the numbers of complex eigenvalues of $\{\lambda_j\}_{j=1}^6$ and $\{\tilde{\lambda}_j\}_{j=1}^6$ are different, we cannot set  $\widetilde{\Gamma}_1=\Gamma_1$. With a  nonsingular matrix $\Theta$, the MATLAB function {\bf fminunc} with $\widetilde{\Gamma}_1$ of form (\ref{gs-11}) as the free parameter matrix is applied to minimize {\rm Rec.OU} in (\ref{gg-r}) with $\tau_1=\tau_2=1$. The updated system satisfies
$$\mbox{Res1.U.b}=3.4990e-14,\ \ \ \ \mbox{Res2.U.b}=4.6628e-16,\ \ \ \ \mbox{Rec.MK}=0.2138.$$

\section*{Acknowledgements}

This work is supported by Research Foundation of Changsha University of Science and Technology under Grant Number 2019QJCZ051 and Department of Education of Hunan Province under Grant number 19B028.  The research is financially supported by Hunan Provincial Key Laboratory of Mathematical Modelling and Analysis in Engineering (Changsha University of Science and Technology, R.P. China).

\section*{References}

\bibliography{mybibfile00}

\begin{thebibliography}{10}
\expandafter\ifx\csname url\endcsname\relax
  \def\url#1{\texttt{#1}}\fi
\expandafter\ifx\csname urlprefix\endcsname\relax\def\urlprefix{URL }\fi
\expandafter\ifx\csname href\endcsname\relax
  \def\href#1#2{#2} \def\path#1{#1}\fi

\bibitem{Sunar-Ps-1999}
M.~Sunar, S.~S. Rao, Recent advances in sensing and control of flexible
  structures via piezoelectric materials technology, Appl. Mech. Rev. 52 (1999)
  1--16.

\bibitem{AAD-2001}
A.~Ahmadian, A.~DeGuilio, Recent advances in the use of piezoceramics for
  vibrations suppression, Shock Vib. Digest 33 (2001) 15--22.

\bibitem{MV-book-smart}
M.~V. Gandhi, B.~S. Thompson, Smart Materials and Structures, Chapman and Hall,
  London, 1992.

\bibitem{IT-book-1990}
T.~Ikeda, Fundamentals of piezoelectricity, Oxford Science Publications,
  Oxford, 1990.

\bibitem{Meng-Ps-2006}
G.~Meng, L.~Ye, X.~Dong, K.~Wei, Closed loop finite element modeling of
  piezoelectric smart structures, Linear Algebra Appl. 13 (2006) 1--12.

\bibitem{Naraya-2003}
S.~Narayanan, V.~Balamurugan, Finite element modeling of piezolaminated smart
  structures for active vibration control with distributed sensors and
  actuators, J. Sound Vib. 262 (2003) 529--562.

\bibitem{Tzou-Ps-1990}
H.~S. Tzou, C.~I. Tseng, Distributed piezoelectric sensor/actuator design for
  dynamic measurement/control of distributed parameter systems: a piezoelectric
  finite element approach, J. Sound Vib. 138 (1990) 17--34.

\bibitem{Qiang-2017}
J.~Qiang, Y.~F. Cai, D.~Chu, R.~Tan, Eigenvalue embedding of undamped
  vaibroacoustic systems with no-spillover, SIAM J. Matrix Anal. Appl. 38~(4)
  (2017) 1190--1209.

\bibitem{Chu-2007}
M.~T. Chu, W.~W. Lin, S.~F. Xu, Updating quadratic modes with no spill-over
  effect on unmeasured spectral data, Inverse Prob. 23 (2007) 243--256.

\bibitem{Chu-2008}
M.~T. Chu, B.~N. Datta, W.~W. Lin, S.~F. Xu, The spill-over phenomenon in
  quadratic model updating, AIAA J. 46 (2008) 420--428.

\bibitem{Chu-2009}
D.~Chu, M.~T. Chu, W.~W. Lin, Quadratic model updating with symmetric, positive
  defiiteness, and no spill-over, SIAM J. Matrix Anal. Appl. 31 (2009)
  546--564.

\bibitem{Kang-2018-apm}
K.~Zhao, L.-Z. Cheng, S.-G. Li, A.-P. Liao, A new updating method for the
  damped mass-spring systems, Appl. Math. Modelling 62 (2018) 119--133.

\bibitem{Lancaster-2008}
P.~Lancaster, Model updating for self adjoint quadratic eigenvalue problems,
  Linear Algebra Appl. 428~(11-12) (2008) 2778--2790.

\bibitem{Friswell-1998}
J.~I. Friswell, J.~E. Mottershead, Finite Element Model Updating in Structural
  Dynamics, Kluwer Academic Publishers, The Netherlands, 1995.

\bibitem{Carvalho-2006}
J.~Carvalho, B.~N. Datta, W.~W. Lin, C.~Wang, Symmetry preserving eigenvalue
  embedding in finite element model updating of vibrating structures, J. Sound
  and Vibration 290 (2006) 839--864.

\bibitem{Lancaster-book}
I.~Gohberg, P.~Lancaster, L.~Rodman, Matrix Polynomialss, SIAM, Academic Press,
  Inc., 1982.

\bibitem{Kang-amc-2014}
K.~Zhao, G.~Yao, Application of the alternating direction method for an inverse
  monic quadratic eigenvalue problem, Appl. Math. Comput. 244 (2014) 32--41.

\bibitem{Kang-2015-COAM}
K.~Zhao, A.-P. Liao, G.-Z. Yao, A proximal point-like method for symmetirc
  finite element model updating problems, Comp. Appl. Math. 34 (2015)
  1251--1268.

\bibitem{Chu-2009-spectral}
M.~T. Chu, S.~Xu, Spectral decomposition of real symmetric quadratic
  $\lambda$-matrices and its applications, Math. Comput. 78 (2009) 293--313.

\bibitem{Golub-book-1996}
G.~H. Golub, C.~V. Loan, Matrix Computation, third edition, Johns Hopkins
  University Press, Baltimore, 1996.

\end{thebibliography}

\end{document}